\documentclass[a4,psfig,english,12pt,epsf,portrait]{article}
\usepackage{amsmath,amsfonts,latexsym,amscd,amssymb,theorem}
\usepackage{epsfig}
\usepackage{color}
\usepackage{amsmath}
\usepackage{graphicx}
\usepackage{bbm}
\def\R{\hbox{\bf R}}

\def\N{\hbox{\bf N}}

\def\p{\partial}

\def\<{\langle}
\def\>{\rangle}
\newcommand{\ba}{\begin{eqnarray}}
\newcommand{\ea}{\end{eqnarray}}

\newtheorem{thm}{Theorem}[section]

\newtheorem{theorem}[thm]{Theorem}
\newtheorem{definition}[thm]{Definition}
\newtheorem{lemma}[thm]{Lemma}
\newtheorem{proposition}[thm]{Proposition}
\newtheorem{corollary}[thm]{Corollary}
\newtheorem{rem}[thm]{Remark}

\newcommand{\eps}{\epsilon}
\numberwithin{equation}{section}

\renewcommand{\R}{{\mathbb R}}

\renewcommand{\N}{{\mathbb N}}

\renewcommand{\p}{\partial}

\begin{document}

\title {SCALAR REDUCTION TECHNIQUES FOR WEAKLY COUPLED HAMILTON--JACOBI SYSTEMS.}
\author{A.Siconolfi\thanks{Mathematics Department, University of Rome "La
Sapienza",  Italy, {\tt siconolf@mat.uniroma1.it}}, \,
S.Zabad\thanks{Mathematics Department, University of Rome "La
Sapienza",  Italy, {\tt zabad@mat.uniroma1.it}}}

\maketitle

\begin{abstract}
We study a class of weakly coupled systems of Hamilton--Jacobi
equations at the critical level.  We associate to it a family of
scalar discounted equation. Using control--theoretic techniques  we
construct an algorithm which allows obtaining a critical solution to
the system as limit of a monotonic sequence of subsolutions.  We
moreover  get a characterization of isolated points of the Aubry set
and  establish semiconcavity properties for critical subsolutions.

\bigskip

\bigskip

\end{abstract}

\section{Introduction}
\parskip +3pt
 This paper deals with  weakly coupled systems of Hamilton-Jacobi equations,  on the torus $\mathbb T^N$, of the form
 \[ H_{i}(x,Du_{i})+\sum_{j=1}^{m}a_{ij}u_{j}(x)=\alpha
 \qquad \mbox{ for every } i \in \{1,\cdots,m\},\]
where $\alpha$ is a real constant and $H_{1}, \cdots,...,H_{m}$ are
continuous Hamiltonians, convex and coercive in the momentum
variable, and $A:=(a_{ij})$ is an $m\times m$ coupling matrix
satisfying
 $$a_{ij}\leq 0 \mbox{ for every }i\neq j, \quad  \sum \limits_{j=1}^{m}a_{ij} = 0 \mbox { for any } i \in \{1,2,...,m\} .$$

 \smallskip

 This is the so--called degenerate case,  because of the
vanishing condition on  any row  of $A$.  It corresponds, for a
single equation, to require it being  of Eikonal type, namely
independent of the unknown function.

\smallskip

 We focus on   the critical system,  obtained taking as
 $\alpha$ the minimum value for which the systems of the above family  admit
viscosity  subsolutions. As first pointed out in \cite{camilli} and
\cite{mitake},  it is the unique system of the family  for which
there are   viscosity solutions on the whole torus.

 The critical setting has been deeply investigated  through PDE techniques in \cite{Davini},  with the purpose  of building an adapted Weak KAM
 theory. It has been in particular proved the existence of an Aubry set,
  denoted in what follows by $\mathcal A$, enjoying   properties similar  to the corresponding object
 in the scalar case. Recently the results of \cite{Davini}
  have been recovered and improved in  \cite{MSTY},
 \cite{ISZ}  by means of a more geometric approach
  to the matter  in a suitable probabilistic frame. In \cite{DSZ}
  the same method has been applied to a time--dependent version of
  the system with the aim of studying an associated
  Lax--Oleinik semigroup.

  There are still several issues to be investigated  and understood
  in the field, especially about the structure and the property of the Aubry set. The major open question being to get regularity
  results for critical subsolutions  on  $\mathcal A$, as in the case of a single equation.

The paper is centered on a method, which seems new, to tackle this
kind of problems:  namely the scalar reduction technique mentioned
in the title. It simply consists in
 associating to the system a family of scalar discounted
 equations. These are roughly speaking obtained by picking one of the   equations in the system, and
freezing all the components of a given critical subsolution except
 the one corresponding to the index of the selected equation. It becomes the unknown of the discounted equation.
 This approach advantageously  allows  exploiting the wide knowledge of
 this kind of
 equations to gather information on the system. We  in particular use that a comparison principle holds for discounted equations, the solutions can be represented as infima of integral functionals, and the  fact that  corresponding optimal
 trajectories do exist.

Our achievements are as follows:   we provide a constructive
algorithm for getting a critical solution by suitably modifying an
initial critical subsolution  outside the Aubry set. The solution is
 obtained as uniform limit of a monotonic  sequence of
subsolutions. The procedure can be useful for numerical
approximation of a critical solution and of the Aubry set as well,
see Remark \ref{important} for more details.

 We moreover give a characterization of isolated points of the
 Aubry, adapting the notion of equilibrium point to systems. This enables us to also show
  the strict differentiability of any critical subsolution on such  points. The final outcome is about   a semiconcavity property for  critical
 subsolutions on the Aubry set, and  on the whole torus for
 solutions. We more precisely prove that the
 superdifferential is nonempty. These results are clearly  related to the aforementioned  open problem
about the differentiability of critical subsolutions
 on the Aubry set.  They can viewed as a partial positive answer to the regularity issue.  We hope they will be useful to fully  crack the
 problem.

 The paper is organized as follows:  Section \ref{notations} contains the notations, some preliminary results and  the setting.
 In Section \ref{control} we introduce the family of scalar discounted equations associated to the critical system,
 and describe an algorithm to get a critical solution starting from
  any subsolution. In Section \ref{applica} we    study, through  the scalar reduction technique,   the nature of isolated points of $\mathcal A$
 and  prove semiconcavity properties for critical subsolutions. Finally, the appendix is
  devoted to recall some results for  critical equations in the scalar case.
 \section{Preliminaries} \label{notations}
 \subsection{Notations  and basic results}

 Throughout the paper, we denote by $\mathbb R^N$ (resp. $\R^m$), $\mathbb{T}^{N}$ the $N$-dimensional Euclidean space and flat torus, respectively, where
 $N$ (resp. $m$) is a positive  integer number.
 We further indicate with $\mathbb R^m_+$ the set made up by the vectors of $\mathbb R^m$
 with nonnegative components.
 We denote by bold characters vectors in $\R^m$ and functions taking values in $\R^m$.  The vector of $\mathbb{R}^{m}$  with  all
 the components equal to $1$ is accordingly denoted by
  $\mathbf {1}$.
  The partial order relations between elements of $\mathbb{R}^{m}$,
denoted by $\leq$,  $<$ must be understood componentwise.
Accordingly, we write $\mathbf u \leq \mathbf v$ (resp. $\mathbf u <
\mathbf v$), for given functions $\mathbf u, \mathbf v
:\mathbb{T}^{N}\rightarrow \mathbb{R}^{m}$, if $\mathbf u(x) \leq
\mathbf v(x)$ (resp. $\mathbf u(x) < \mathbf v(x)$)  for every $x
\in \mathbb{T}^{N} $.

\smallskip

 Given an upper semicontinuous (respectively, lower semicontinuous) function $f$ on $\mathbb{T}^{N}$ or on $\R^N$, we say that a continuous  function $\phi$ is supertangent
 (res. subtangent) to $f$ at some point $x_{0}$ if $\phi(x_{0})=f(x_{0})$ and $\phi \geq
 f$ (resp. $\phi \leq f$)
 in some neighborhood  of $x_{0}$. The set made up by  the differentials of $C^1$ functions supertangent (resp. subtangent) to $f$ at $x_0$ is called
 the  superdifferential (resp. subdifferential) of $f$ at $x_0$, and
 is indicated with
$D^{+}f(x_0)$ (resp. $D^{-}f(x_0)$).

If $f$  is locally Lipschitz continuous function then it is almost
everywhere differentiable with locally bounded gradient,
 in force of Rademacher's Theorem. The (Clarke) generalized gradient of $f$, denoted by $\partial{f}$, is defined for every $x$ via the formula
$$\partial{f}(x)=\mathrm {co}\{p=\lim_{n}D f(x_{n}), f \mbox{ is differentiable at }x_{n}, \lim_{n}x_{n}=x\},$$
where $\mathrm {co}(\cdot)$ denotes the convex hull. The function
$f$ is called strictly differentiable at  $x$ if $f$ is
differentiable at $x$ and $D f$ is continuous at x. This is
equivalent to require $\partial f(x)$ to be a singleton, namely
$\partial f(x)=\{Df(x)\}$. We recall from \cite{DSZ}, Lemma 1.4, a
result on the a.e. derivative  of $f$ along an absolutely continuous
curve

\smallskip

\begin{lemma}\label{clarke} Let  $\eta: (-\infty,0]  \to \mathbb T^N$ be an  absolutely continuous curve. Let $s$ be such that $t \mapsto
f\big(\eta(t)\big)$ and $t \mapsto \eta (t)$ are both differentiable
at $s$. Then
\[ \frac d{d t} f(\eta(t))_{\mbox{\Large $|$}_{t=s}} =  p
\cdot \dot \eta(s) \quad\hbox{for some
 $p \in \partial f(\eta(s))$.}\]
\end{lemma}

\smallskip

We refer  to \cite{capuzzo} or \cite{Barles} for  the definition of
viscosity solution,and more generally for a comprehensive treatment
of viscosity solutions theory.

\smallskip

We record for later use
 \smallskip
 \begin{proposition} \label{testlipfct}
 Let  $u$ be an USC subsolution (resp. LSC supersolution) of a Hamilton--Jacobi equation of the form
\[F(x,u,Du)=0 \qquad\hbox{in $\mathbb T^N$,}\]
where $F$ is continuous in all arguments and nondecreasing in $u$.
Let $\psi$ be a
 Lipschitz-continuous supertangent (resp. subtangent) to $u$ at some point $x_0$. Then
 $$ F(x_0, u(x_0),p) \leq 0 \quad \hbox{(resp. $\geq 0$)} \quad \mbox{ for some } p \in \partial \phi(x).  $$
\end{proposition}\begin{proof} Given $\varepsilon >0$, we define the
$\varepsilon$--inf--convolution of $\psi$ via
\[\psi_\varepsilon(x)=\min_{y \in \mathbb T^N}\left(\psi(y)+\frac{1}{2\varepsilon}|y-x|^2\right). \] We can
 assume, without loosing generality,  $\psi$ to be strict supertangent, and so $x_0$ to be
the unique maximizer of $u -\psi$ in a suitable closed ball $B$
centered at $x_0$.  From this uniqueness property we deduce that any
sequence $x_\eps$ of maximizers of $u-\psi_\eps$ in $B$ converges to
$x_0$. Hence $x_\eps$ is in the interior of $B$ for $\eps$
sufficiently small, and then for such $\eps$, $\psi_\eps$ is
supertangent to $u$ at $x_\eps$. The inf--convolution being
semiconcave,  we deduce  that
\begin{equation}\label{liptestb1}
    H(x_\eps, u(x_\eps), p_\eps) \leq 0 \qquad\hbox{for any $p_\eps \in
\p\psi_\eps(x_\eps)$.}
\end{equation}
  Further
\[ u(x_\eps)- \psi_\eps(x_\eps) = \max_B u-\psi_\eps \to \max_B
u-\psi= u(x_0)- \psi(x_0),\] which, implies, bearing in mind that
$\lim \psi_\eps(x_\eps) = \psi(x_0)$
\begin{equation}\label{liptestb0}
    \lim_{\eps \to 0} u(x_\eps)=u(x_0).
\end{equation}
 It is also well known  that for any $y_\eps$ realizing the minimum
 in the formula defining $\psi_\eps$ at
 $x_\eps$, we have
\[\p\psi_\eps(x_\eps) \cap D^+\psi(y_\eps)  \subset
\p\psi_\eps(x_\eps) \cap \p\psi(y_\eps) \neq \emptyset.\] Taking
into account that $y_\eps \to x_0$, as $\eps$ goes to $0$,
exploiting \eqref{liptestb0},
 plus the continuity properties of
 $F$ and generalized gradients, we
find for $q_\eps \in \p\psi_\eps(x_\eps) \cap \p\psi(y_\eps)$
\begin{eqnarray*}
q_\eps &\to&  p \in \p\psi(x) \\
F(x_\eps,u(x_\eps),q_\eps)&\to& F(x, u(x),p)  \\
\end{eqnarray*}
then, thanks to \eqref{liptestb1}
\[F(x,u(x),p) \leq 0 \quad\hbox{and} \quad p \in \p\psi(x)\]
as claimed.
 $\hfill{\Box}$
\end{proof}

\subsection{Setting of the problem}\label{system}
  Here we introduce the system under investigation as well as some basic assumptions and preliminary
  facts. We basically follow the approach of \cite{Davini}.

  \smallskip

   We deal with a one--parameter family of  weakly coupled systems of Hamilton-Jacobi equations of the form
\begin{equation}\label{wcs} \tag{HJ$\alpha$}
H_{i}(x,Du_{i})+\sum_{j=1}^{m}a_{ij}u_{j}(x)=\alpha \quad \mbox{in } \mathbb{T}^{N} \quad \mbox{ for every } i \in \{1,\cdots,m\},
\end{equation}

\noindent where $\alpha$ is a real constant and $H_{1},
\cdots,...,H_{m}$ are Hamiltonians satisfying, for any $i \in
\{1,\cdots,m\}$, the following set of assumptions:
\begin{itemize}
\item[(H1)] $ H_{i}:\mathbb{T}^{N}\times\mathbb{R}^{N}\rightarrow\mathbb{R} \quad \mbox{is continuous} ;$
 \item[(H2)] $p\mapsto H_{i}(x,p) \quad \mbox{is convex for every } x \in \mathbb{T}^{N}; $
  \item[(H3)]$p\mapsto H_{i}(x,p) \quad \mbox{is coercive for every } x \in \mathbb{T}^{N}. $
 \end{itemize}
The above conditions will be assumed throughout the paper  without
further mentioning. For some specific results we will also need the
following additional requirements for  $i \in \{1,2,...,m\}$:
\begin{itemize}
 \item[(H4)] $p\mapsto H_{i}(x,p) \quad \mbox{is strictly convex for every } x \in \mathbb{T}^{N}
 $;
  \item[(H5)] $(x, p)\mapsto H_{i}(x,p) \quad \mbox{is locally Lipschitz continuous in $\mathbb T^N \times \mathbb R^N$ .}$
 \end{itemize}

The $m\times m$ coupling matrix $A:=(a_{ij})$   satisfies:

\begin{itemize}
\item[(A1)]$a_{ij}\leq 0 \mbox{ for every }i\neq j$;
\item[(A2)]$\sum \limits_{j=1}^{m}a_{ij} =0 \mbox { for any } i \in \{1,\cdots,m\};$
\item[(A3)]  is irreducible, i.e for every $W \subsetneq \{1,2,...,m\}$ there exists $i \in W $ and $j \notin W$ such that $a_{ij}<0.$
\end{itemize}

The relevant consequence  of   the irreducibility condition  in our
context is
 that the system cannot split into independent subsystems. Under our assumptions the coupling matrix $A$ is singular with
rank  $m-1$ and kernel spanned by $\mathbf {1}$. Moreover
$\mathrm{Im} A$, which has dimension  $m -1$, cannot contain vectors
with strictly positive or negative components. This  in particular
implies that $\mathrm{Im}(A) \cap \mathrm{ker}(A)=\{0\}$.

\medskip

We also derive:
\begin{proposition}
All the diagonal elements of the coupling matrix  $A$ are strictly
positive.
\end{proposition}
\begin{proof}
It is clear  that $a_{ii} \geq 0$ for any $i$. Indeed, if $a_{kk}=0$
for some $k \in \{1,2,...,m\}$, then condition $(A2)$  would imply
$a_{kj}=0$ for every $j \in \{1,2,...,m\}$, which contradicts the
irreducibility character of the matrix. $\hfill{\Box}$
\end{proof}

\medskip

The notion of viscosity (sub/super) solution  can be easily adapted
to systems as (\ref{wcs}).
\begin{definition}\textbf{(Viscosity solution)}.\\
 \noindent We say that a continuous function $\mathbf u:\mathbb{T}^{N}\rightarrow\mathbb{R}^{m}$ is a viscosity subsolution of (\ref{wcs}) if for every $(x,i) \in \mathbb{T}^{N} \times \{1,2,...,m\} $, we have$$H_{i}(x,p)+\sum_{j=1}^{m}a_{ij}u_{j}(x)\leq \alpha \quad \mbox{for every } p \in D^{+}u_{i}(x).$$

 \noindent Symmetrically, $\mathbf u$ is a viscosity supersolution of (\ref{wcs}) if for every
  $(x,i) \in \mathbb{T}^{N} \times \{1,2,...,m\} $, we have$$H_{i}(x,p)+\sum_{j=1}^{m}a_{ij}u_{j}(x)
  \geq \alpha \quad \mbox{for every } p \in D^{-}u_{i}(x).$$

\noindent  Finally, if $\mathbf u$ is both a viscosity sub and
supersolution, then it is called a viscosity solution.
\end{definition}

From now on,  we will drop the term viscosity since no other kind of
weak solutions will be considered. As already pointed out in the
Introduction, we will be especially interested in the critical
weakly coupled system
\begin{equation}\label{criticalwcs} \tag{HJ$\beta$}
H_{i}(x,Du_{i})+\sum_{j=1}^{m}a_{ij}u_{j}(x)= \beta \quad \mbox{in }
\mathbb{T}^{N} \quad \mbox{ for every } i \in \{1,\cdots,m\},
\end{equation}

where
\[\beta = \min \{ \alpha \in \mathbb R \mid
\;\hbox{\eqref{wcs} admits subsolutions}\}\]

\smallskip

We straightforwardly derive form the coercivity condition:
\smallskip

\begin{proposition}\label{equilip} Let $\alpha \geq \beta$. The family of all  subsolutions to \eqref{wcs} is
equi-Lipschitz continuous with Lipschitz constant denoted by
$\ell_\alpha$.
\end{proposition}

\medskip

 Moreover, owing to the convexity of the
Hamiltonians, the notion of viscosity and a.e. subsolutions are
equivalent for \eqref{wcs}. Furthermore, we can express the same
property using generalized gradients of any component. Namely, $u$
is a subsolution to \eqref{wcs} if and only if
$$ H_i(x,p) + \sum_{j=1}^{m}a_{ij}u_{j}(x) \leq \alpha $$
for any $x \in \mathbb{T}^{N}$, $p \in \partial u_{i}(x)$, $i \in
\{1, \cdots, m\}$.

\medskip

Due to the fact that any subsolution of (\ref{criticalwcs}) is
$\ell_\beta$--Lipschitz continuous,   we may modify the $H_i$'s
outside the compact set $\{(x,p): |p| \leq \ell_\beta\}$, to obtain
a new Hamiltonian which is still continuous and convex, and in
addition  satisfies superlinearity condition, for every $i$. Since
the sublevels contained in $B(0,\ell_\beta)$ are not affected, the
subsolutions of the system obtained by replacing the $H_i$'s in
(\ref{criticalwcs}) by the new
Hamiltonians are the same as the original one.\\
In the remainder of the paper, we will therefore assume without any
loss of generality
\begin{itemize}
    \item[(H'2)] $H_{i}$ is superlinear in $p$ for any
$i \in \{1,2,...,m\}$,
\end{itemize}

 We can thus  associate to any $H_i$ a
Lagrangian function $L_{i}$  through the Fenchel transform, i.e.
$$L_{i}(x,q)=\sup _{p \in \mathbb{R}^{N}}\{pq-H_{i}(x,p)\}.$$
The function $L_{i}$ is continuous on $\mathbb{T}^{N} \times
\mathbb{R}^{N}$, convex and superlinear in $q$.

\bigskip

As pointed out in \cite{Davini},  there is   restriction in the
value that a subsolution  to the system \eqref{wcs} can assume at a
given point of the torus.  An adaptation of the pull-up method used
in the scalar version of the theory gives:

 \begin{proposition}\label{maxsub}
 Let $\alpha \geq \beta$. The maximal subsolution of (\ref{wcs}) among those taking the same admissible  value at a given point $y$ is solution to (\ref{wcs})
  in $\mathbb{T}^{N} \setminus \{y\}$.
  \end{proposition}

\smallskip

As mentioned  in the Introduction, the Aubry set associated to the
critical system, denoted by $\mathcal{A}$, is a closed nonempty
subset of $\mathbb{T}^{N}$ where the obstruction of getting global
strict subsolution concentrates, see \cite{Davini}, \cite{ISZ}. More
specifically, there cannot be any critical subsolution which is, in
addition, locally strict at a point in $\mathcal{A}$, in the sense
of the subsequent definition.

\medskip

\begin{definition} Given a subsolution $\mathbf u$ of (\ref{criticalwcs}). We say that $u_{i}$ is   locally strict at $y \in \mathbb{T}^{N}$ if
there exists a neighborhood $U$ of $y$ and $\delta >0$ such that
 $$H_{i}(x,Du_{i}(x)) + \sum_{j=1}^{m}a_{ij}u_{j}(x) \leq \beta - \delta \quad \mbox{ for a.e. } x \in U. $$
 We  say that $u_i$  is strict in an open subset $U$ of $\mathbb T^N$  if it is locally strict at any $y\in U$.
 We say that $\mathbf u$ is locally strict at $y$  or strict in an open subset of the torus, if the propriety holds for any
 component.
 \end{definition}

 \medskip

 An useful criterion  is the following:

\begin{lemma}\label{criteria} Let $y \in \mathbb{T}^{N}$ and $\mathbf u$ be a subsolution of (\ref{criticalwcs}).
The $i-$th component of $\mathbf u$ is locally strict at $y$ if and only if
$$ H_{i}(y,p)+\sum_{j=1}^{m}a_{ij}u_{j}(y) < \beta \quad \mbox{ for any } p \in \partial u_i(y).$$
\end{lemma}

 \medskip

 The formal PDE definition of Aubry set is:

 \begin{definition}\textbf{(Aubry set)}.\\
 \noindent The Aubry set $\mathcal A$ is made up by points $y$ such
 that any maximal subsolution  of the critical system among those
 taking a given admissible value at $y$ is a solution on the whole
 torus.
 \end{definition}
 We refer to \cite{ISZ} for a more  geometric characterization of
 $\mathcal A$  through cycles, in a suitable probabilistic framework.
The following results on the Aubry set are taken from \cite{Davini}.
They generalize to systems facts already known in the case of a
single equation.

\smallskip

 \begin{theorem} \label{t22}  A point $y \in {\mathcal A}$ if
 and only if  no critical subsolution is  locally strict at $y$.
\end{theorem}

\medskip

\begin{theorem}\label{strictstrict} There exists a critical subsolution  to the system
which is strict $\mathbb T^N \setminus \mathcal A$.
\end{theorem}

\medskip

We conclude the section recalling a final  property which  is
peculiar to systems and has no equivalent in the scalar case. There
is  a remarkable rigidity phenomenon taking place,
namely the restriction on admissible values that a critical
subsolution can attain at a given point. This becomes severe on the Aubry
set. We have

\smallskip
\begin{theorem}\label{t4}
Let $y \in {\mathcal A}$ and $\mathbf u, \mathbf v$ be two
subsolutions of (\ref{criticalwcs}), then
\begin{equation}
\mathbf u(y)-\mathbf v(y)=k\mathbf {1}, \quad k \in \mathbb{R}
\end{equation}
\end{theorem}



  \section{Scalar reduction  } \label{control}

In this section we associate to the critical  system   some
discounted scalar equations. Using these equations, we will
thereafter write an algorithm  to construct a solution of the
critical system by suitably modifying outside $\mathcal A$ a given
critical subsolution.

\smallskip
We denote by $\mathbf w=(w_{1},\cdots,w_{m})$ the initial  subsolution of
(\ref{criticalwcs}) and freeze all its components except one
obtaining  for a given $i \in \{1, \cdots, m\}$,  the following
discounted equation:
\begin{equation} \label{e9}
a_{ii}v(x) + H_{i}(x,Dv)+\sum_{j\neq i}a_{ij}w_{j}(x)-\beta=0, \quad
\hbox{in $\mathbb{T}^{N}$.}
\end{equation}
For simplicity we set $f_{i}(x)=-\sum_{j\neq i}a_{ij}w_{j}(x)+\beta,$ for every $i$.
\medskip

\medskip
The discounted equation satisfies a comparison principle.  This is
well known, however the proof in our setting  is simplified to some
extent by exploiting the compactness of the ambient space plus the
coercivity  of the Hamiltonian with respect to the momentum
variable. This straightforwardly implies that all subsolutions are
Lipschitz--continuous and allows using  Proposition
\ref{testlipfct}. We provide the argument for reader's convenience.

\smallskip

\begin{theorem}\label{t12} If $u$, $v$
 an USC continuous subsolution  and a LSC supersolution of (\ref{e9}), respectively, then $u \leq v$ in $\mathbb{T}^{N}$.
\end{theorem}
\begin{proof}
We recall that  $u$ is Lipschitz continuous. Let $x_{0}$ be a point
in $\mathbb{T}^{N}$ where $v-u$ attains its minimum and assume, for
purposes of contradiction, that
\begin{equation} \label{e29}
v(x_{0})-u(x_{0}) < 0.
\end{equation}
The function $u$ is therefore a Lipschitz continuous subtangent to
$v$ at $x_{0}$ and hence we have
$$a_{ii} v(x_{0}) +H_{i}(x_{0},p)-f_{i}(x_0) \geq 0 \quad \mbox{ for some } p \in \partial {u(x_{0})}.$$
But $u$ is a viscosity subsolution of (\ref{e9}), then
$$a_{ii} u(x_{0}) +H_{i}(x_{0},p)-f_{i}(x_0)  \leq 0 \quad \mbox{ for all } p \in \partial {u(x_{0})}.$$
Subtracting the above two inequalities, we get
$$a_{ii} (v(x_{0})-u(x_{0}))\geq 0,$$
which contradicts (\ref{e29}). We therefore conclude that
$\min_{\mathbb{T}^{N}}(v-u) \geq 0$, which in turn implies $v \geq
u$ in $\mathbb{T}^{N}$, as desired. $\hfill{\Box}$
\end{proof}

\medskip

The equation   (\ref{e9}) can be interpreted as the
Hamilton--Jacobi--Bellman equation of a control problem with the
Lagrangian $L_i$ as cost and $a_{ii}$ as discount factor. The
control is given by the velocities which are in principle unbounded, but this is somehow compensated by the coercive character
of the Hamiltonian. The corresponding value function
 $v^i:\mathbb{T}^{N} \rightarrow \mathbb{R}$  is defined  by
\begin{equation} \label{e10}
v^i(x)=\inf_{\gamma} \int_{-\infty}^{0} e^{a_{ii}s} \left(L_{i}(\gamma(s),\dot{\gamma}(s))+f_{i}(\gamma(s))\right)ds,
\end{equation}
where the infimum is taken over all absolutely continuous curves $
\gamma :]-\infty,0] \rightarrow \mathbb{T}^{N}$ with $\gamma(0)=x$.

\smallskip

We have:

\begin{theorem}\cite[Appendix 2 ]{davini2}\label{t7} The discounted value function $v^i$ is the unique continuous viscosity solution of (\ref{e9}).
 Moreover, for every $x \in \mathbb{T}^{N}$ there exists a curve $\gamma : (-\infty,0] \rightarrow \mathbb{T}^{N}$ with $\gamma(0)=x$ such that
\begin{equation} \label{e19}
v^i(x)= \int_{-\infty}^{0} e^{a_{ii}s} \left(L_{i}(\gamma(s),\dot{\gamma}(s))+f_{i}(\gamma(s))\right)ds.
\end{equation}
\end{theorem}

\bigskip

\begin{corollary} \label{r1}   We have
\begin{equation}\label{r1-1}
   w_{i} \leq v^i \mbox{    in } \mathbb{T}^{N} \quad\hbox{and}\quad w_{i} = v^i \mbox{    in }
\mathcal A.
\end{equation}
\end{corollary}
\begin{proof} Since  $w_{i}$ is a viscosity subsolution of (\ref{e9}), we derive
from   Theorem \ref{t12} the inequality in \eqref{r1-1}. Taking into
account that $a_{ij} \leq 0 $ for every $i \neq j$, we further
derive that the vector valued function obtained from $w$ by
replacing $w_i$ with $v^i$ and keeping all other components
unaffected, is  still a subsolution of (\ref{criticalwcs}). The
equality in \eqref{r1-1} then comes from the rigidity phenomenon in
${\mathcal A}$,  see Theorem \ref{t4}.
 $\hfill{\Box}$
\end{proof}

\medskip

\subsection{Algorithm}

In this subsection  we will construct a monotonic sequence of
critical subsolutions $(\mathbf v_{n})$ which converges, up to a
subsequence, to a solution of (\ref{criticalwcs}).

\bigskip

\noindent {step 1: Construction of the sequence} $(\mathbf v_{n})$.\\
\\
Let  $\mathbf{w}={\mathbf v_0}=(v_0^{1},v_0^{2},...,v_0^{m})$  be any subsolution of (\ref{criticalwcs}).\\
The first element $\mathbf v_{1}=(v_{1}^{1},v_{1}^{2},...,v_{1}^{m})$ of $(\mathbf v_{n})$ is defined  component by component as follows :\\ \\
 for $k= 1 \cdots m$,  $v_{1}^{k}$ is the solution of the discounted
equation
$$H_{k}(x,Du)+a_{kk}u+ \sum_{j < k} a_{kj}v_{1}^{j}(x)+\sum_{j >k}a_{kj}v_0^{j}(x)=\beta,$$
where the possible empty sums in the above formula are counted as
$0$. By construction, see the proof of Corollary \ref{r1}, $\mathbf v_{1}$ is
a critical subsolution of the system. In addition, using Corollary
\ref{r1}, we get
$$\mathbf w =\mathbf v_0 \leq \mathbf v_{1} \quad \mbox{and} \quad \mathbf  w=\mathbf v_0=\mathbf v_{1} \mbox{ on }\mathcal{A} .$$
 \\
 We iterate the above procedure  to construct  $\mathbf v_n$, for any $n \in \N$, $n >1$,  starting from $\mathbf v_{n-1}$.

\smallskip
 We get that any element $\mathbf v_n$ is a critical subsolution of the
 system and
 \begin{equation}\label{prop1}
    \mathbf v_{n-1} \leq \mathbf v_n  \qquad\hbox{for any $n$}
\end{equation}
\begin{equation}\label{prop2}
   \hbox{all the $\mathbf v_n$ coincide   on $\mathcal{A}.$}
\end{equation}

\noindent \textbf{step 2: Convergence of the sequence} $(v_{n})$. \\ \\
 We exploit Proposition \ref{equilip} to infer that the functions $(\mathbf v_{n})$ are equi-Lipschitz.
 Moreover, all the  $(\mathbf v_{n})$'s  take a fixed value on the Aubry set,  see \eqref{prop2}, so they are equibounded as well.
  We derive by Ascoli-Arzel\`{a} Theorem that  $(\mathbf v_{n})$ converge  uniformly, up to
  subsequences, and we in turn deduce  convergence of the whole sequence
  because of its monotonicity, see \eqref{prop1}. \\ \\
 \noindent \textbf{step 3: proving the limit is a critical solution }.\\ \\
 We denote by  $\mathbf V=(V_{1},V_{2},...,V_{m})$ the uniform limit of $\mathbf v_n$.\\
 Given $k \in \{1, \cdots,m \}$, we have, by construction,  that  $v_n^{k}$ is the solution of
 $$F_{n}^k(x,u,Du):=H_{k}(x,Du)+a_{kk}u+\sum_{j< k}a_{kj}v_n^{j}(x) + \sum_{j> k}a_{kj}v_{n-1}^{j}(x)=\beta.$$
The Hamiltonians $F^k_n$ converge  uniformly  in $\mathbb T^N \times
\R \times \mathbb R^N$, as $n \to+ \infty$,  to
 $$F^k(x,u,p):=H_{k}(x,p)+a_{kk}u+\sum_{j\neq k} a_{kj}V_{j}(x).$$
 Consequently,  by basic stability  properties in viscosity solutions theory,  $V_{k}$ is solution to the limit equation
 $$F^k(x,u,Du)=H_{k}(x,Du)+a_{kk}u+\sum_{j\neq k}a_{kj}V_{j}(x)=\beta.$$
  We conclude that  the limit $\mathbf V=(V_{1},V_{2},...,V_{m})$ is  solution of (\ref{criticalwcs}), as it was claimed.
\medskip

\medskip

\begin{rem}\label{important} i) As already pointed out in the Introduction,  the  above algorithm  can have  a numerical interest  to compute  critical solutions
of the system via  the analysis of a sequence of scalar discounted
equations. The latter problem has been extensively studied and well
tested  numerical codes are available. It is clearly required the
knowledge of a critical subsolution as starting point, but this is
 easier than the determination of a solution. In addition, since the
 initial subsolution is not affected on $\mathcal A$ at any step of
 the procedure, the algorithm can be useful to get an approximation
 of the Aubry set itself.
 ii) In principle  we could apply the algorithm also starting from any
supercritical subsolution. What happens is that the sequence we
construct is not any more anchored at the Aubry set, and we get in
the end a sequence of functions positively diverging at any point.
We believe that the rate of divergence could be exploited to
estimate  how far the supercritical value we have chosen is from
$\beta$, but we do not investigate any further this issue in the
present paper.
\end{rem}


  \section{Applications of the scalar reduction}\label{applica}

 True to title, we describe in this section  two applications of the scalar reduction method. Namely, we provide a characterization of the isolated points of $\mathcal A$ and
 establish  some semiconcavity properties of critical subsolutions to the system.
To do that, we need a strengthened form of Theorem \ref{t7}  for
points belonging to $\mathcal A$.

\smallskip

\medskip
\begin{proposition} \label{prop20}   Let $y \in {\mathcal A} $ and $\mathbf u$ be a
subsolution of (\ref{criticalwcs}). Then for every $i \in \{1,
\cdots,m\}$ there exists a curve $\gamma : (-\infty,0] \rightarrow
\mathbb{T}^{N}$ with $\gamma(0)=y$ such that
$$u_i(y)= \int_{-\infty}^{0} e^{a_{ii}s} \left(L_{i}(\gamma(s),\dot{\gamma}(s))-
\sum_{j\neq i}a_{ij}u_{j}(\gamma(s))+\beta\right)ds,$$
and $\gamma(t) \in {\mathcal A} $ for every $t$.

\end{proposition}

For the proof we need a preliminary lemma

\begin{lemma} \label{prop36} Let $\mathbf u$ be a subsolution of
(\ref{criticalwcs}) strict outside ${\mathcal A} $. Then for every
$y \in \mathcal{A}$, there exists a critical solution $\mathbf v$
such that
$$\mathbf u(y)=\mathbf v(y) \quad \mbox{ and } \quad  u < v \mbox{ on } \mathbb{T}^{N}\setminus\mathcal{A}.$$
\end{lemma}
\begin{proof}   Given $y \in \mathcal{A}$,  we consider the maximal critical subsolution $\mathbf v$ taking the value $\mathbf u(y)$ at $y$.
Then, by the very definition of Aubry set, $\mathbf v$ is a critical
solution,  and   $\mathbf v \geq u$.  If this inequality were not
strict at some $x_{0} \in \mathbb{T}^{N}\setminus\mathcal{A}$, then
$u_i(x_0)=v_i(x_0)$ for some index $i$, and consequently  $u_{i}$
should be subtangent to $v_{i}$ at $x_{0}$, and hence by Proposition
\ref{testlipfct}
$$\beta \leq H_{i}(x_0,p)+\sum_{j=1}^{m}a_{ij}v_{j}(x_0) \leq H_{i}(x_0,p)+\sum_{j=1}^{m}a_{ij}u_{j}(x_0) $$
for some $p \in \partial u_{i}(x_{0})$. This  contradicts  $u_{i}$
being locally strict at $x_0$, in view of Lemma \ref{criteria}.

$\hfill{\Box}$
\end{proof}

\smallskip

\begin{proof} { \bf of the Proposition \ref{prop20}}   We consider a  critical subsolution  $\mathbf w$ to the system strict outside
 $\mathcal A$, see Theorem \ref{strictstrict}.
It is not restrictive, by adding a suitable constant,
 to assume $u_i(y)= w_i(y)$, where $\mathbf u$ is the subsolution appearing in the statement. This  in turn implies by the rigidity property on the Aubry set, see Theorem
 \ref{t4},  $\mathbf u(y)=\mathbf w(y)$.
 We in addition denote by $\bar u$ the
maximal subsolution taking the value $\mathbf w(y)=\mathbf u(y)$ at
$y$. It is a critical solution to the system in view of  Theorem
\ref{t22} and, according to Lemma \ref{prop36},  we also have
\begin{equation} \label{e26}
 w < \bar u \quad\mbox{ on } \mathbb{T}^{N}\setminus\mathcal{A}.
\end{equation}
Now, let $\underline {v}$ be the solution of the discounted equation
$$H_{i}(x,Dv)+a_{ii}v(x)+\sum_{j\neq i}a_{ij}w_{j}(x)=\beta,$$
and $\bar{v}$ the solution of
$$H_{i}(x,Dv)+a_{ii}v(x)+\sum_{j\neq i}a_{ij} \bar u_{j}(x)=\beta.$$
We deduce from Corollary \ref{r1}
\begin{equation}\label{empoli1}
   \bar{v}(y)=\underline{v}(y)=\bar u_i(y)=u_i(y)=w_i(y).
\end{equation}
 There exists, in force of Theorem \ref{t7}, a curve $\gamma :
(-\infty,0] \rightarrow \mathbb{T}^{N}$ with $\gamma(0)=y$ such that
$$\bar u_i(y)= \int_{-\infty}^{0} e^{a_{ii}s} \left(L_{i}(\gamma(s),\dot{\gamma}(s))-\sum_{j\neq i}a_{ij}\bar u_{j}(\gamma(s))+\beta\right)ds.$$
Assume, for purposes of contradiction, that the support of $\gamma$
is not contained in  $\mathcal{A}$ then, taking into account
(\ref{e26}),  that $a_{ij} < 0$ for $i \neq j$ by (A1) plus
irreducibility of $A$, we get
\begin{eqnarray*}
  w_i(y)&\leq& \int_{-\infty}^{0} e^{a_{ii}s} \left(L_{i}(\gamma(s),\dot{\gamma}(s))-\sum_{j\neq i}a_{ij}w_{j}(\gamma(s))+\beta\right)ds \\
   &<& \int_{-\infty}^{0} e^{a_{ii}s} \left(L_{i}(\gamma(s),\dot{\gamma}(s))-\sum_{j\neq i}a_{ij}\bar u_{j}(\gamma(s))+\beta\right)ds \\
   &=& \bar u_i(y),
\end{eqnarray*}
which is impossible in view of \eqref{empoli1}.  By the maximality
property of $\bar u$, we also have
\begin{eqnarray*}
 u_i(y)&\leq& \int_{-\infty}^{0} e^{a_{ii}s} \left(L_{i}(\gamma(s),\dot{\gamma}(s))-\sum_{j\neq i}a_{ij}u_{j}(\gamma(s))+\beta\right)ds \\
   &\leq& \int_{-\infty}^{0} e^{a_{ii}s} \left(L_{i}(\gamma(s),\dot{\gamma}(s))-\sum_{j\neq i}a_{ij}\bar u_{j}(\gamma(s))+\beta\right)ds \\
   &=& \bar u_i(y),
\end{eqnarray*}
which proves that $\gamma$ is also optimal for $u_i(y)$ and
 concludes the proof.  $\hfill{\Box}$
\end{proof}

\subsection{Equilibria}  In this section we provide a characterization of the isolated  points of $\mathcal A$. To this aim, we  introduce the notion of
equilibrium points of the weakly coupled system, which we compactly
write in the form
  \begin{equation} \label{equilibriumwcs}
\mathbf{H}(x,D{u})+Au=\beta \mathbf {1},
\end{equation}
where the Hamiltonian $\mathbf{H}:\mathbb{T}^{N} \times
\mathbb{R}^{mN}$ has the separated variable form
$$\textbf{H}(x,p_1,\cdots,p_m)=(H_1(x,p_1), \cdots, H_m(x,p_m)).$$
We consider the equilibrium distribution $\textbf{o} \in \R^m$ which
is uniquely identified by the following conditions:
\begin{itemize}
  \item[1)] $\mathbf{o}\,A=0$
  \item[2)] $\mathbf{o}\cdot \mathbf {1}=1$
\end{itemize}
 It is an immediate consequence of the condition $\mathrm{Im} \,
A\cap \mathbb{R}^{m}_{+}=\{0\}$  plus $\mathrm {dim \,Im} A=m-1$,
that all the vectors orthogonal to $\mathrm{Im}\, A$ have either
strictly positive or strictly negative components.
Consequently $\textbf{o}$ is a probability vector, i.e. all its components are nonnegative and sum up to $1$.\\
Multiplying the system (\ref{equilibriumwcs}) by $\textbf{o}$ we get
that all subsolutions $\mathbf u$ satisfy
\begin{equation}\label{equi1}
   \textbf{o} \cdot \textbf{H}(x, p_1, \cdots,p_m) \leq \beta \qquad\hbox{for any $x$, $p_i \in \partial u_i(x)$.}
\end{equation}
For $x \in \mathbb{T}^{N}$, we  set $\min\limits_{p}\textbf{H}(x,p):=\left (\min \limits_{p}H_{1}(x,p), \cdots, \min \limits_{p}H_{m}(x,p) \right  )$.\\
Then we deduce from \eqref{equi1} that
$$\textbf{o} \cdot \min_{p}\textbf{H}(x,p) \leq \beta \qquad\hbox{for any $x$}.$$
We call a point $x$ equilibrium if $$\textbf{o} \cdot \min \limits_{p}\textbf{H}(x,p)=\beta.$$
\smallskip
We see from the above definition that if $x$ is an equilibrium and
$\mathbf u$ a critical subsolution then for any $i$  and  $q  \in
\partial u_i(x)$  we have
\begin{eqnarray}
  H_i(x,q) &=& \min_p H_i(x,p)  \label{equi2}\\
  \beta &=& H_{i}(x,q)+\sum_{j=1}^{m}a_{ij} \, u_j(x) \label{equi3}
\end{eqnarray}
 This implies that any
subsolution also satisfies the supersolution property at an
equilibrium point, and we deduce from  Proposition \ref {maxsub} and
Theorem \ref{t22}  that the set of equilibria is contained in the
Aubry set.

The next proposition is a partial converse of  this fact, it
provides a characterization of the isolated points of Aubry set,
which is a generalization of the scalar case.
 \smallskip
\begin{proposition} \label{isolated}
Any isolated point of the Aubry set is an equilibrium.
\end{proposition}
\begin{proof} Let $x$ be an isolated point of Aubry set and $\mathbf u$ be
a critical subsolution of the system. Then, in view of Proposition
\ref{prop20}, there exists a curve $\gamma$ with $\gamma(0)=x$ such
that
$$u_{i}(x)= \int_{-\infty}^{0} e^{a_{ii}s} \left(L_{i}(\gamma(s),\dot{\gamma}(s))-\sum_{j\neq i}a_{ij}u_{j}(\gamma(s))+\beta \right) ds, \quad \mbox{ for every } i \in \{1, \cdots,m\}$$
and  the support of $\gamma$ is contained in  ${\mathcal A} $.
Exploiting the fact that $x$ is isolated, we get $\gamma(t)\equiv x$
for every $t$ and hence
$$u_{i}(x)= \frac{1}{a_{ii}} \,\left (L_i(x,0)-\sum_{j\neq i}a_{ij}u_{j}(x)+ \beta \right ), \quad  \mbox{ for every } i \in \{1, \cdots,m\} .$$
Then
\begin{eqnarray*}
 \nonumber A\mathbf u(x) &=& ( L_1(x,0),\cdots,L_m(x,0)) + \beta \mathbf {1} \\
  \nonumber &=& -\min\limits_{p}\textbf{H}(x,p)+ \beta \mathbf {1}.
\end{eqnarray*}
Multiplying by $\textbf{o}$ and taking into account that
$\textbf{o}$ is a probability vector orthogonal to $\mathrm{Im}(A)$,
we get
$$\textbf{o} \cdot \min \limits_{p}\textbf{H}(x,p)=\beta$$
as desired.
$\hfill{\Box}$
\end{proof}
\medskip

Assuming  the strict convexity assumption (H4),  we get a
regularity result.

\smallskip
\begin{proposition}
Under the additional assumption (H4), any critical subsolution is
strictly differentiable at every isolated point of $\mathcal{A}$.
\end{proposition}
\begin{proof}
Let $x_0$ be an isolated point of $\mathcal{A}$ and $\mathbf u$ be a
subsolution of (\ref{criticalwcs}). Then for every $i \in \{1,
\cdots,m\}$, we have
$$H_{i}(x_0,p_i)+\sum_{j=1}^{m}a_{ij}u_{j}(x_0)\leq \beta \quad \mbox{for every } p _i\in \partial u_{i}(x_0).$$
This implies
\[
\sum_{i=1}^{m}o_{i}H_i(x_0,p_i) \leq \beta.
\]
Taking into account that  $x_0$ is equilibrium we deduce from the
above inequality that
$$\beta=\sum_{i=1}^{m}o_{i}\min _{p}H_i(x_0,p) \leq \sum_{i=1}^{m}o_{i}H_i(x_0,p_i) \leq \beta,$$
which in turn gives that
$$H_i(x_0,p_i)=\min _{p}H_i(x_0,p), \quad \mbox{ for every } p _i\in \partial u_{i}(x_0) \mbox{, } i \in \{1, \cdots,m\}.$$
Due to  $H_i$ being strictly convex,  the above minimum is unique
and hence $\partial u_{i}(x_0)$ reduces to a singleton. This implies
strict differentiability of $\mathbf u$ at $x_0$. $\hfill{\Box}$
\end{proof}


  \subsection{Semiconcavity-type  properties for critical subsolutions}\label{estimate}

In this section we study   a family of Eikonal equations derived
from the critical system.   The main information we gather through
this approach, under the additional assumptions (H4),(H5),  is that
the superdifferential of any critical solution of
(\ref{criticalwcs})
 is nonempty at every point of the torus. The same property holds true for any  critical subsolution on the Aubry set. \\

\bigskip
We start by stating and proving a consequence of Theorem \ref{t7}.

\begin{proposition}\label{optcu} Let $\mathbf u$, $x$, $i$ be a critical solution to the system,  a point
in $\mathbb T^N$ and an index in $\{1, \cdots,m\}$, respectively.
There is a curve $\gamma$ defined in $(- \infty,0]$   such that
$\gamma(0)=x$ and
\[ \frac d{dt} u_i(\gamma(t))  =  L_i(\gamma(t),\dot\gamma(t))- \sum_{j} a_{ij} \,
u_j(\gamma(t)) +\beta   \qquad \hbox{for a.e. $t \in
(-\infty,0)$.}\]
\end{proposition}

\begin{proof}   Taking into account that $u_i$ is the solution of the discounted equation \eqref{e9} with $u_j$ in place of $w_j$,
we know by  Theorem \ref{t7} that  there  is an optimal curve
$\gamma$ defined in $(-\infty,0]$ with $\gamma(0)=x$ such that
\begin{equation}\label{optcu0}
   u_i(x)= \int_{-\infty}^{0} e^{a_{ii}s}
\left [ L_{i}(\gamma(s),\dot{\gamma}(s))-\sum_{j\neq i}
a_{ij}u_{j}(\gamma(s))+\beta \right ] ds.
\end{equation}
 We claim that $\gamma$
also satisfies the  statement of the proposition. We define
\[ g(t)= e^{a_{ii} t} u_i(\gamma(t))  \qquad\hbox{for $t \in (-\infty,0)$,}\]
accordingly
\begin{equation}\label{optcu1}
   \frac d{dt} \,  g(t) = a_{ii} \, e^{a_{ii} t} u_i(\gamma(t)) + e^{a_{ii}
t} \, p(t) \cdot \dot\gamma(t)
\end{equation}
 for a.e. $t$, where $p(t)$ is a
suitable element of $\partial u_i(\gamma(t))$ satisfying $\frac
d{dt} u_i(\gamma(t))=p(t) \cdot \dot\gamma(t)$ for a.e. $t$, see
Lemma \ref{clarke}. We further get taking into account that $\mathbf
u$ is a solution to the critical system
\begin{equation}\label{optcu2}
    p(t) \cdot \dot\gamma(t) \leq H_i(\gamma(t), p(t)) + L_i(\gamma(t), \dot\gamma(t)) \leq -\sum_{j}
a_{ij}u_{j}(\gamma(t))+\beta + L_i(\gamma(t), \dot\gamma(t)).
\end{equation}
We derive from  \eqref{optcu0}, \eqref{optcu1}, \eqref{optcu2}
\begin{eqnarray*}
 u_i(x) &=& \lim_{t \to - \infty} g(0)- g(t) = \int_{-\infty }^0 \frac d{dt} \, g(t) \, dt  \\
   &\leq& \int_{-\infty }^0 a_{ii} \, e^{a_{ii} t} u_i(\gamma(t)) - e^{a_{ii} t} \, \left  (\sum_{j}
a_{ij}u_{j}(\gamma(t) ) -\beta -
   L_i(\gamma(t), \dot\gamma(t)) \right  )\, dt \\
   &=& \int_{-\infty }^0 e^{ a_{ii} t} \, \left [ L_i(\gamma(t),\dot\gamma(t)) - \sum_{j \neq i} a_{ij} \,
u_j(\gamma(t)) + \beta \right ] \, dt =u_i(x).
\end{eqnarray*}
This in turn implies
\[\frac
d{dt} u_i(\gamma(t))= p(t) \cdot \dot\gamma(t) = -\sum_{j}
a_{ij}u_{j}(\gamma(t))+\beta + L_i(\gamma(t), \dot\gamma(t))
\qquad\hbox{for a.e. $t$,}\]
 as it was to be proved. $\hfill{\Box}$
\end{proof}

\medskip

In the case where the point $x$ belongs in addition to $\mathcal A$,
we get, thanks to Proposition  \ref{prop20}, a strengthened  form of
the previous assertion.

\begin{corollary} \label{coroptcu}  The statement  of Proposition \ref{optcu} holds true for  any
critical subsolution $\mathbf u$, provided   $x \in \mathcal A$. The
curve $\gamma$ is in addition contained in $\mathcal A$.
\end{corollary}

\begin{proof}  If $\mathbf u$ is any critical subsolution, we know from Proposition \ref{prop20} that there is an optimal curve $\gamma$ for $u_i(x)$
which is in addition  contained in $\mathcal A$. We then  prove that
$\gamma$ satisfies the assertion arguing as in Proposition
\ref{optcu}.
\end{proof}

\medskip

We recognize that the integrand appearing in the statement of
Proposition \ref{optcu} is nothing but the Lagrangian associated
through Fenchel transform to the Hamiltonian
\begin{equation} \label{e11}
H_{i}^{\mathbf u}(x,p)= H_{i}(x,p)+\sum_{j=1}^{m}a_{ij}u_{j}(x).
\end{equation}
Given a critical subsolution $\mathbf u$ to the system,   we
therefore consider the Eikonal equation
\begin{equation}\label{e14}
H_{i}^{\mathbf u}(x,Dv)=\beta \qquad\hbox{in $\mathbb T^N$,}
\end{equation}
and denote by $\sigma^{\mathbf u}_i$, $S_i^{\mathbf u}$ the
corresponding support function and intrinsic distance, respectively,
given by suitably adapting \eqref{support} and \eqref{distance}.
Since $u_i$ is a subsolution to  \eqref{e14}, it is clear that the
critical value of $H_i^{\mathbf u}$ is less than or equal to
$\beta$. We in addition have:

\smallskip

\begin{proposition}\label{optcubis} The critical value of $H^{\mathbf u}_i(x,p)$ is equal to $\beta$,
for any critical subsolution $\mathbf u$ to the system, any index $i
\in \{1,\cdots, m\}$. In addition the limit points, as $t \to -
\infty$,of any curve satisfying the statement of Proposition
\ref{optcu}/Corollary \ref{coroptcu} belong to the corresponding
Aubry set.
\end{proposition}
\begin{proof} We fix $\mathbf u$ and $i$. Let us consider $x \in \mathcal A$ and an optimal curve $\gamma$
as in the statement with $\gamma(0)=x$.  We denote by
$y$  a limit point   of $\gamma$ as $t \to - \infty$. If  the set of
such limit points reduces to $y$,  then there is a sequence $t_n \to
- \infty$ with
\[\frac{d}{dt} u_i(\gamma(t_n)) =
L_i(\gamma(t_n),\dot\gamma(t_n))- \sum_{j} a_{ij} \,
u_j(\gamma(t_n)) +\beta \quad\hbox{ and} \quad \dot \gamma (t_n) \to
0,\] therefore
\[0 = \lim_{t_n \to - \infty} \frac{d}{dt} u_i(\gamma(t_n))= \lim_{t_n \to -
\infty} L_i(\gamma(t_n),\dot\gamma(t_n))- \sum_{j} a_{ij} \,
u_j(\gamma(t_n)) +\beta.\] By continuity of $L_i$, $u_j$  we deduce
\[ L_i(y,0)- \sum_{j} a_{ij} \,
u_j(y)= -\beta\] or equivalently    $\min \limits_p H^u_i(y,p)=
\beta$. Since we know that $\beta$ is supercritical for
$H_i^{\mathbf u}$, this implies that $\beta$ is actually the
critical value of $H^{\mathbf u}_i$ and $y$ belongs to the
corresponding Aubry set, by Proposition \ref{prop13}.

If instead the limit set of $\gamma$, as $t \to - \infty$, is not a
singleton, then we find $\gamma(t_n)$ converging to $y$ such that
the curves $\gamma_n:=\gamma \big |_{[t_n,t_{n+1}]}$ possess
Euclidean length bounded from below by a positive constant. We have
\begin{eqnarray*}
  \int_{t_n}^{t_{n+1}} \sigma^{\mathbf u}_i(\gamma(s),\dot\gamma(s)) \,ds &\leq& \int_{t_n}^{t_{n+1}} \left [ L_i(\gamma(s),\dot\gamma(s))-
\sum_{j} a_{ij} \, u_j(\gamma(s)) +\beta \right ] \, ds \\
   &=& u_i(\gamma(t_{n+1})) - u_i(\gamma(t_n)).
\end{eqnarray*}
We deduce  from Proposition \ref{dire} that the leftmost inequality
in the above formula must actually be an equality.  This  shows that
the intrinsic length $\int_{t_n}^{t_{n+1}} \sigma^{\mathbf
u}_i(\gamma(s),\dot\gamma(s)) \,ds$  is infinitesimal as $n \to +
\infty$.

 We  construct a sequence of cycles $\eta_n$ based on
$y$ by concatenating  the segment linking $y$ to $\gamma (t_n)$,
$\gamma_n$ and the segment linking $\gamma(t_{n+1})$ to $y$. We find
that the intrinsic  lengths of such cycles are infinitesimal, as $n
\to + \infty$, while the Euclidean lengths stay bounded from below
by a positive constant.  Taking into account the very definition of
Aubry set for scalar Eikonal equations, we derive also in this case
that $\beta$ is the critical value of $H^u_i$, and $y$ belongs to
the corresponding Aubry set. This concludes the proof.
$\hfill{\Box}$
\end{proof}

\bigskip

\noindent We denote by $\mathcal A_{i}^{\mathbf u}$  the Aubry set
associated with $H_{i}^{\mathbf u}$ at the critical level $\beta$,
for $i \in \{1,\cdots,m\}$.

\smallskip

\begin{proposition} \label{prop6} We have that
\begin{equation}\label{prop6-1}
 \mathcal A_i^{ \mathbf u} \cap \mathcal A \neq \emptyset \qquad\hbox{for any
susbsolution $\mathbf u$ to \eqref{criticalwcs}, \,any $i$.}
\end{equation}
 If, in addition,
$\mathbf u$ is  strict on $\mathbb{T}^{N} \setminus \mathcal A$,
then
\begin{equation}\label{prop6-2}
  \mathcal A_{i}^{\mathbf u} \subseteq \mathcal A \quad \mbox{ for every } i \in
\{1,\cdots,m\}
\end{equation}
\end{proposition}
\begin{proof}
Formula \eqref{prop6-1} is a direct consequence of Corollary
\ref{coroptcu}  and Proposition \ref{optcubis}.  To show
\eqref{prop6-2}, let us consider  $y \notin \mathcal A $, then
$u_{i}$ is locally strict at $y$ for every $i \in \{1,\cdots,m\}$.
Hence, there exists an open neighborhood $W$ of $y$ and $\delta > 0
$ such that
$$H_{i}(x,Du_{i}(x))+\sum_{j=1}^{m}a_{ij}u_{j}(x) < -\delta+\beta \quad \mbox{ for a.e. } x \in W,  \quad \mbox{ for every } i \in \{1,\cdots,m\}.$$
Therefore, $u_{i}$ is a critical subsolution of (\ref{e14}) which is
locally  strict at $y$ and consequently $y \notin \mathcal
A_{i}^{\mathbf u}$.
\end{proof}
$\hfill{\Box}$

\bigskip

As already announced,    we assume (H4), (H5) to establish  the
final result. Note that, due to the Lipschitz character of any
subsolution to the system, the Hamiltonians $H^u_i$  are locally
Lipschitz--continuous in $\mathbb T^N \times \R^N$, for any
subsolution $\mathbf u$ of \eqref{criticalwcs}, any index $i$.
\smallskip

In this setting we obtain:

\begin{theorem}\label{prop8}
We assume(H4), (H5). If $\mathbf u$ is a critical subsolution of
(\ref{criticalwcs}), then
$$D^{+}u_{i}(x) \neq \emptyset \qquad \mbox{ for every } i \in \{1,2,...,m\}, \, x \in \mathcal A.$$
If, in addition, $\mathbf u$ is a solution to (\ref{criticalwcs})
then the above property holds true for any $x \in \mathbb T^N$.
\end{theorem}
\begin{proof}
First assume $\mathbf u$ to be subsolution of (\ref{criticalwcs}).
If $x_0 \in \mathcal A^{\mathbf u}_i$  then  $u_i$ is differentiable
at $x_0$, according to Proposition \ref{prop13}. This proves the
assertion. If instead $x_0 \in  \mathcal A \setminus \mathcal
A^{\mathbf u}_i$, then we derive from the proof of Proposition
\ref{optcubis} that
\begin{equation}\label{prop8-1}
   u_{i}(x_0) \geq  \min \limits_{y \in \mathcal A_{i}^{\mathbf u}}\{u_{i}(y)+S_i^{\mathbf u}(y,x_0) \} .
\end{equation}
By Proposition \ref{maxsubclosed}, the function on the right
hand--side of the above formula
 is the maximal subsolution to \eqref{e14} with trace
$ u_i$ on $\mathcal A_i^{\mathbf u}$, this  implies that equality
must prevail in \eqref{prop8-1}.
 There is  then an element $y_0
\in \mathcal A_{i}^{\mathbf u}$ such that \[u_{i}(x_{0})=
u_{i}(y_0)+S_{i}^{\mathbf u}(y_0,x_{0}).\] Hence
$u_{i}(y_0)+S_{i}^{\mathbf u}(y_0,.)$ is supertangent to $u_{i}$ at
$x_{0}$, and so by Proposition \ref{siusa} the superdifferential of
$u_i$ is nonempty at $x_0$, as it was claimed.  If $\mathbf u$ is in
addition solution of (\ref{criticalwcs}), the same argument of above
gives that $D^+u_i(x_0) \neq \emptyset$ at any $x_0\in \mathbb T^N$.
This concludes the proof. $\hfill{\Box}$
\end{proof}

\begin{appendix}

  \section{Scalar Eikonal equations} \label{eikonal}
  In this appendix  we consider a continuous Hamiltonian $H:\mathbb{T}^{N}\times\mathbb{R}^{N}\rightarrow\mathbb{R}$
  satisfying the following assumptions:
  \begin{itemize}
 \item[(E1)] $p\mapsto H(x,p) \quad \mbox{is convex for every } x \in \mathbb{T}^{N};$
  \item[(E2)]$p\mapsto H(x,p) \quad \mbox{is superlinear for every } x \in \mathbb{T}^{N}. $
 \end{itemize}

 For some results we need the following additional  conditions:
\begin{itemize}
 \item[(E3)] $p\mapsto H(x,p) \quad \mbox{is strictly convex for every } x \in \mathbb{T}^{N};$
  \item[(E4)]$(x,p) \mapsto H(x,p) \quad \mbox{is locally Lipschitz--continuous in  } \mathbb{T}^{N} \times \R^N. $
 \end{itemize}

 \smallskip

 We consider the family of Hamilton-Jacobi equations
\begin{equation} \label{e12}
H(x,Du)=a \quad \mbox{in }  \mathbb{T}^{N},
\end{equation}
where $a$ is a real parameter.  We introduce an intrinsic
semidistance $S_{c}$ on the torus, related to \eqref{e12},  via
\begin{equation} \label{distance}
S_{a}(x,y)=\inf_\xi \left \{\int_{0}^{1} \sigma_a
(\xi(s),\dot{\xi}(s))ds \right  \},
\end{equation}
where  $\xi$ varies in the family of absolutely continuous curve
taking the value $x$, $y$ at $s=0$ and $s=1$, respectively,  and
$\sigma_a(x,q)$ is the support function of the $a$-sublevel of $H$,
namely
\begin{equation}\label{support}
   \sigma_a(x,q):=\max\{p \cdot q: H(x,p)\leq a\}.
\end{equation}
The following inequality holds true
\begin{equation}\label{eiko1}
    \sigma_a(x,q) \leq L(x,q)+a \qquad \hbox{for any $x \in \mathbb
    T^N$, $q \in \R^N$,}
\end{equation}
where $L$ is the Lagrangian associated to $H$ via the Fenchel
transform thanks to (E2).

\smallskip
\begin{proposition}\label{dire}  A function $u$ is subsolution to \eqref{e12}  if and
only if
\[u (x)-u(y) \leq S_{a}(y,x) \quad \mbox{ for every }\; x,y \in \mathbb T^N.\]

\end{proposition}

\medskip

It is well known that there exists a unique value of $a$,  denoted
by $c$ and called critical, for which the equation (\ref{e12}) has a
solution on the whole torus. It is defined as
$$c=\inf\{a: \mbox{ }(\ref{e12}) \mbox{ has a subsolution}\}.$$ We then focus on  the critical equation
\begin{equation} \label{e27}
H(x,Du)=c.
\end{equation}
In what follows we will recall some important facts and results
about \eqref{e27}, taken from \cite{siconolfi}, that will play an
important role for the semi-concavity  results of Section
\ref{estimate}.

\smallskip

\begin{proposition}\label{siusa} Under the additional assumptions (E3),(E4) the
function
\[ x \mapsto S_c(y,x)\]
possess nonempty superdifferential  at $x$ for any $y$ and  any $x
\neq y$.
\end{proposition}

\bigskip

In the analysis of the behavior of critical subsolutions, a special
role is played by the so--called Aubry set, denoted by  $\mathcal
A_{e}$,   defined as the collection of points $y \in \mathbb{T}^{N}$
such that
$$\inf_\xi \left \{\int_{0}^{1} \sigma_c (\xi(s),\dot{\xi}(s))ds  \right \}=0,$$
where $\xi$ varies in the class of absolutely continuous cycles
based on $y$ with Euclidean length uniformly estimated from below by
some positive constant. It is not difficult to see that the
definition is independent of such constant. In the next statement we
recall some relevant properties of $\mathcal A_{e}$.

\smallskip

\begin{proposition} \label{prop13} \hfill
\begin{itemize}
\item[(i)] Any point $y$ with $\min \limits_p H(y,p) = c$ belongs to
$\mathcal A_e$;
\item[(ii)] A point $y \notin \mathcal A_{e}$ if and only if there exists a critical
subsolution which is locally strict at $y$;
\item[(iii)] Under the additional assumptions (E3),(E4),  any subsolution of
\eqref{e27} is strictly differentiable at every point of $\mathcal
A_e$.
\end{itemize}
\end{proposition}
\smallskip

We  finally record:
\begin{proposition}\label{maxsubclosed}
Given   a continuous function  $u_0$   defined on a  closed set $C
\subset \mathcal A_e$ such that
\[u_{0}(x)-u_{0}(y) \leq S_{c}(y,x) \quad \mbox{ for every }\; x,y \in C ,\]
 then the function
\[u:= \min \limits_{y \in C} \{u_{0}(y)+S_c(y,.) \} \]
is the maximal subsolution of (\ref{e27}) in $\mathbb{T}^{N}$
agreeing with  $u_{0}$ on $C$, and is a solution  as well.
\end{proposition}
\end{appendix}



\end{document}